\documentclass[article]{amsart}
\usepackage{enumerate}
\usepackage{enumitem}
\usepackage{amsfonts,amssymb,amsmath,amsthm}
\usepackage{epsfig}
\usepackage{graphics}
\usepackage[normalem]{ulem}
\usepackage{color}
\usepackage{comment}
\usepackage{stmaryrd} 
\usepackage{url}
\usepackage{overpic}

\input xy 
\xyoption{all}
\numberwithin{equation}{section}

\definecolor{OrangeRed}{cmyk}{0,0.6,1,0}            
\definecolor{DarkBlue}{cmyk}{1,1,0,0.20}
\definecolor{DarkGreen}{cmyk}{1,0,0.6,0.2}
\definecolor{myblue}{rgb}{0.66,0.78,1.00}
\definecolor{Violet}{cmyk}{0.79,0.88,0,0}
\definecolor{Lavender}{cmyk}{0,0.48,0,0}

\newtheorem{thm}{Theorem}[section]
\newtheorem{theorem}[thm]{Theorem}
\newtheorem{main theorem}[thm]{Main Theorem}
\newtheorem{corollary}[thm]{Corollary}

\newtheorem{prop}[thm]{Proposition}

\theoremstyle{definition}

\newtheorem{definition}[thm]{Definition}
\newtheorem{remark}[thm]{Remark}

\newtheorem{example}[thm]{Example}

\DeclareMathAlphabet{\mathbbmsl}{U}{bbm}{m}{sl}

\def\C{\mathbb C}
\def\P{\mathbb P}

\def\bcases{\begin{cases}}

\def\ecases{\end{cases}}

\newcommand{\N}{\mathbb N}

\newcommand{\bea}{\begin{eqnarray*}}
\newcommand{\eea}{\end{eqnarray*}}

\newcommand{\be}{\begin{equation}}
\newcommand{\ee}{\end{equation}}

\newcommand{\Jac}{\operatorname{Jac}}

\newcommand{\MM}{\mathcal{M}}
\newcommand{\RR}{\mathcal{R}}

\renewcommand{\epsilon}{\varepsilon}
\renewcommand{\phi}{\varphi}

\newcommand{\xbf}{{\bf x}}
\newcommand{\ebf}{{\bf e}}
\newcommand{\obf}{{\bf 0}}
\newcommand{\zbf}{{\bf z}}
\newcommand{\ybf}{{\bf y}}
\newcommand{\Ybf}{{\bf Y}}
\newcommand{\ubf}{{\bf u}}
\newcommand{\ibf}{{\bf 1}}
\newcommand{\F}{\mathbbmsl{F}}

\begin{document}

\title{The independence polynomial on recursive sequences of graphs}

\author{Mikhail Hlushchanka}
\address{Korteweg-de Vries Instituut voor Wiskunde, Universiteit van Amsterdam,  1090 GE \newline Amsterdam, The Netherlands}
\email{mikhail.hlushchanka@gmail.com}

\author{Han Peters}
\address{Korteweg-de Vries Instituut voor Wiskunde, Universiteit van Amsterdam,  1090 GE \newline Amsterdam, The Netherlands}
\email{h.peters@uva.nl}

\begin{abstract}
    We study the zero sets of the independence polynomial on recursive sequences of graphs. We prove that for a maximally independent starting graph and a stable and expanding recursion algorithm, the zeros of the independence polynomial are uniformly bounded. Each of the recursion algorithms leads to a rational dynamical system whose formula, degree and the dimension of the space it acts upon depend on the specific algorithm. Nevertheless, we demonstrate that the qualitative behavior of the dynamics exhibit universal features that can be exploited to draw conclusions about the zero sets.
\end{abstract}

\maketitle

\section{Introduction}

\subsection{Motivation and main result.} In statistical physics, one typically considers the normalized limit of graph polynomials for a sequence of graphs $(G_n)_{n\ge 0}$ that converge in an appropriate sense to an infinite limit graph. Examples include the free energy of the partition function for the Ising, hard-core, or Potts models on increasing subgraphs of a regular lattice. Besides computing the actual values of these graph polynomials, which is often infeasible for large graphs, a primary interest is a qualitative description of the limiting free energy, including a characterization of real parameters where the free energy is non-analytic, the so-called phase transitions. Obtaining such qualitative information is notoriously difficult for many graph polynomials, even for seemingly simple limit graphs such as the square lattice.

\medskip

There have been a number of successful studies in the last decade where the authors have considered sequences of graphs that do not converge to a regular lattice but are instead recursively defined. Notable examples of such studies include the partition function of the Ising model and the chromatic polynomial on the diamond hierarchical lattice~\cite{BLRI, BLRII, chio2021chromatic}, the hard-core model on Cayley trees~\cite{Rivera, Willigen}, and the Ising model on Sierpi\'{n}sky gasket graphs and Schreier graphs~\cite{DDN2011}, and the spectrum of the Laplace operator on graphs corresponding to self-similar groups~\cite{Bac2023}. In all of these studies, the recursive nature of the graph sequences leads to an iterative scheme for the corresponding graph polynomials. By analyzing such an induced dynamical system,
the authors in~\cite{BLRI, BLRII, chio2021chromatic, Rivera} successfully described the zero sets of the graph polynomials and as a consequence the (non-)analyticity of the limiting free energy.

\medskip

The zeros of graph polynomials are also essential in understanding the computational complexity of the evaluation of graph polynomials, see \cite{PatelRegts}. In particular, such zeros for sequences of recursive graphs played a crucial role behind the scenes in the results from \cite{GalanisInapproximability,PetersRegts}, which focus on the complexity of approximating the independence polynomial values for bounded degree graphs.

\medskip

This paper aims to establish a unified framework for describing graph polynomials on sequences of recursive graphs. We focus specifically on one type of graph polynomial --  the independence polynomial, also known as the partition function of the hard-core model. We introduce quite general (recursive) algorithms for generating recursive graph sequences and analyze the resulting zero sets.

Informally, our algorithm works in the following way. Starting with a graph $G_0$ with $k$ marked vertices labeled $1, \ldots, k$, we iteratively construct a sequence of graphs $(G_n)_{n\geq 0}$, each with $k$ marked vertices. Namely, at each step we take $m$ copies of the previous graph $G_n$, connect some of the $m\cdot k$ marked vertices with identical labels in these copies, and afterward assign $k$ new marked vertices in the resulting graph (in a pre-described manner) to obtain the graph $G_{n+1}$.

\begin{theorem}[Main result]
The graph recursion $\RR$ defined by any abstract gluing data $(H,\Sigma, \Upsilon, \Phi)$ with parameters $m,k\geq 2$ induces a dynamical system $F: \C^{(2^k)}\to \C^{(2^k)}$ given by a homogeneous polynomial map of degree $m$. The resulting rational map $\widehat F$ on $\P^{(2^k-1)}$ leaves invariant a manifold $\MM$ of dimension $k$; this manifold depends only on the reduced gluing data $(H, \Phi)$.

Furthermore, if the gluing data is stable and expanding, then there is a submanifold $\MM_0\subset \MM$ and an iterate $n$ such that $\MM_0$ consists of fixed points for $\widehat F^n$ and it is normally superattracting.

As a consequence, the zeros of the independence polynomials of a recursive graph sequence $(G_n)_{n\geq 0}$, $G_{n+1}=\RR(G_n)$, are bounded for any maximally independent starting graph $G_0$.
\end{theorem}

The graph recursion with respect to a gluing data will be formally defined later in the introduction, along with the notions of stability, expansiveness, and maximally independence.

\subsection{Independence polynomial}\label{ss:independence-poly} In this paper, unless otherwise stated, by a \emph{graph} $G$ we always mean a multi-graph, that is, $G$ is given by a pair $(V(G), E(G))$, where $V(G)$ is a finite \emph{vertex set} and $E(G)$ is a finite \emph{edge multi-set} of undordered pairs in $V(G)$.  As usual, the \emph{vertex degree} of $v\in V(G)$
is the number of adjacent edges, that is, the number of pairs $\{v,w\}\in E(G)$.
We call any map $\sigma: V(G)\to \{0,1\}$ a \emph{vertex assignment} on $G$. We say that such $\sigma$ is \emph{independent} if $\sigma^{-1}(1):=\sigma^{-1}(\{1\}) \subseteq V(G)$ forms an \emph{independent set} in $G$, that is, if no two vertices in $\sigma^{-1}(1)$ form an edge of $G$.

The (\emph{univariant}) \emph{independence polynomial} of $G$ is defined as
$$
Z_G(\lambda):=\sum_{\sigma} \lambda^{\#\sigma^{-1}(1)},
$$
where the sum is taken over all independent vertex assignments on $G$. The independence polynomial corresponds to the partition function of the \emph{hard-core model}, a model from statistical physics commonly used to model gasses.

The zeros of the independence polynomials $Z_G(\lambda)$ play an important role in both the physical models and in questions regarding the computational hardness of estimating the value of $Z_G(\lambda)$ at a given parameter $\lambda$. Given a sequence of graphs $(G_n)_{n \ge 0}$, we define the \emph{limiting free energy} by
$$
\rho(\lambda) = \lim_{n \rightarrow \infty} \frac{\log|Z_{G_n}(\lambda)|}{\#V(G_n)},
$$
provided this limit exists. It is said that the graph sequence $(G_n)_{n \ge 0}$ exhibits a \emph{phase transition} at a real parameter $\lambda_0 >0$ when the real function $\rho(\lambda)$ is not real analytic at $\lambda_0$. We emphasize that formally we always consider phase transitions of a \emph{sequence of graph} rather than that of a limiting object such as an infinite graph or a metric space. Results of Lee-Yang \cite{YangLee, LeeYang} showed that for suitable sequences of graphs, absence of zeros near $\lambda_0$ implies the absence of a phase transition.

An important topic in theoretical computer science concerns the effective computation of graph polynomials for large graphs. It turns out that the exact computation of partition functions is almost always \emph{$\#P$-hard}. The question of whether partition functions can be effectively approximated, up to some multiplicative error, has received extensive attention in the recent literature, see for example the related paper~\cite{GalanisInapproximability}. A relationship between the absence of \emph{complex} zeros and the existence of polynomial algorithms for the approximation of partition functions was shown by Patel and Regts \cite{PatelRegts}.

The computational hardness and the description of the zero sets for the independence polynomial has been extensively studied for the set of all graphs whose vertex degrees are bounded by some uniform constant $\Delta \ge 2$. It was shown by Regts and the second author \cite{PetersRegts} that complex zeros are absent in a neighborhood of the real interval between two critical parameters $\lambda_- < 0 < \lambda_+$, which implies the existence of a polynomial time algorithm for approximating $Z_G(\lambda)$ for $\lambda\in (\lambda_-,\lambda_+)$. On the other hand, it was shown by Bezakova et.\ al.\ \cite{BezakovaEtAl,BoerEtAlApprox} that zeros are dense outside of a bounded domain $U_\Delta$, which relates to approximation of $Z_G(\lambda)$ being $\#P$-hard outside of $U_\Delta$. Both of these results rely on understanding the behavior of a dynamical system induced by different sequences of recursive graphs.

We end this discussion by mentioning the recent papers \cite{HelmuthPerkinsRegts,BoerEtAlTorus} in which the authors discuss the boundedness of zeros of the independence polynomial for different sequences of graphs converging to cubic lattices. These results are closely related to what we prove here for recursive sequences of graphs.

\subsection{Recursive definition}\label{ss:rec-definition}

Let us formally describe an abstract recursive procedure for producing sequences of graphs that we study in this paper. Each graph in such a sequence $(G_n)_{n \ge 0}$ will have $k\ge 2$ distinct marked vertices, labeled $1,\dots, k$. A recursive step will consist of taking $m \ge 2$ copies of the previous graph and connecting some of the copied marked vertices according to a specific rule, while also marking and labeling $k$ vertices of the resulting graph. This construction will be governed by a \emph{gluing data} $(H, \Sigma, \Upsilon, \Phi)$, where $H$ will specify which marked vertices of the $m$ copies are joined, the pair $(\Sigma, \Upsilon)$ will specify how they are joined, and $\Phi$ will specify how the new marked vertices are assigned.

\begin{definition}[Gluing data $(H, \Sigma, \Upsilon, \Phi)$]\label{defn: formal data}
A (\emph{formal}) \emph{gluing data} (with parameters $m\geq 2$ and $k\geq 2$) is a quadruple $(H, \Sigma, \Upsilon, \Phi)$  where each item is specified as follows:
\begin{itemize}
    \item $H$ is a multi-hypergraph with the vertex set $V(H)=\{w_1, \ldots, w_m\}$ and edge multi-set $E(H)$. This means that the edges of $H$ are given by non-empty subsets $e \subseteq V(H)$, where a single subset may occur multiple times in $E(H)$. We assume that each vertex of $H$ is  contained in exactly $k$ edges, and that each of these edges $e$ has a distinct label $\ell(e)\in\{1,\ldots, k\}$. In other words, for each $j\in \{1,\ldots, k\}$, the edges of $H$ labeled $j$ specify a partition of $V(H)=\{w_1,\ldots,w_m\}$. We call the labeled multi-graph $H$ the \emph{gluing scheme}.

    \item $\Sigma=(\Sigma_e)_{e\in E(H)}$ is a collection of (non-empty) connected multi-graphs, which we call \emph{connecting graphs}. We assume that each $\Sigma_e$ contains a ``special'' marked vertex $v_e\in V(\Sigma_e)$, which we call the \emph{root} of $\Sigma_e$.


    \item $\Upsilon = (\Upsilon_e)_{e\in E(H)}$ is a collection of (not necessarily injective) maps $\Upsilon_e: e \to V(\Sigma_e)$, which we call \emph{attaching maps}.

    \item Finally, $\Phi: \{1, \ldots, k\} \rightarrow E(H)$ is an injective map, which we call the \emph{labeling map}.
\end{itemize}
We call the pair $(H,\Phi)$ as above the \emph{reduced gluing data} (with parameters $m,k$).
\end{definition}

We now specify the recursive step governed by formal gluing data.

\begin{definition}[Graph recursion $\RR$] Let $(H, \Sigma, \Upsilon, \Phi)$ be a gluing data with parameters $m\geq 2$ and $k\geq 2$ as above, and suppose that we are given a graph $G$ with $k$ distinct marked vertices $v_1, \ldots, v_k$. Here and below, we think of each vertex $v_j$ as having the label $j$. We now define a new marked graph $\RR(G)=\RR_{(H, \Sigma, \Upsilon, \Phi)}(G)$ as follows:
\begin{enumerate}[label=(\Roman*)]
    \item First, we take $m$ copies $G(1),\dots, G(m)$ of $G$, where each copy $G(i)$ is associated to the vertex $w_i$ of the gluing scheme $H$.

    \item Then for each edge $e = \{w_{i_1}, \ldots, w_{i_s}\} \in E(H)$ with label $\ell(e)\in\{1,\dots,k\}$, we connect together the marked vertices labeled $\ell(e)$ in the copies $G(i_1),\ldots,G(i_s)$ by identifying them with the vertices $\Upsilon_e(w_{i_1}),\ldots,\Upsilon_e(w_{i_s})$ in $\Sigma_e$, respectively. We denote the obtained graph $\RR_{(H,\Sigma,\Upsilon)}(G)$. We observe that, by construction, the edges of $H$ containing only a single vertex do not lead to connections between vertices of different copies of $G$.

    \item Finally, we assign $k$ marked vertices in the graph $\RR_{(H,\Sigma,\Upsilon)}(G)$ according to the labeling map $\Phi$: the label $j \in \{1, \ldots, k\}$ is given to the vertex of $\RR_{(H,\Sigma,\Upsilon)}(G)$ that corresponds to the root of $\Sigma_{e}$, where $e=\Phi(j) \in E(H)$.
\end{enumerate}
The resulting marked graph, which we denote $\RR(G)=\RR_{(H, \Sigma, \Upsilon, \Phi)}(G)$, is said to be obtained from $G$ according to the gluing data $(H, \Sigma, \Upsilon, \Phi)$.
\end{definition}

In this way, given a gluing data $(H, \Sigma, \Upsilon, \Phi)$ with parameters $m, k$ and an arbitrary starting graph $G_0$ with $k$ distinct marked vertices labeled $1,\dots, k$, we produce a sequence $(G_n)_{n\geq 0}$ of marked graphs defined by $G_{n+1} := \RR_{(H, \Sigma, \Upsilon, \Phi)}(G_0)$, each with $k$ distinct marked vertices labeled $1,\dots, k$.


\begin{figure}[t]
\begin{overpic}[width=\textwidth]{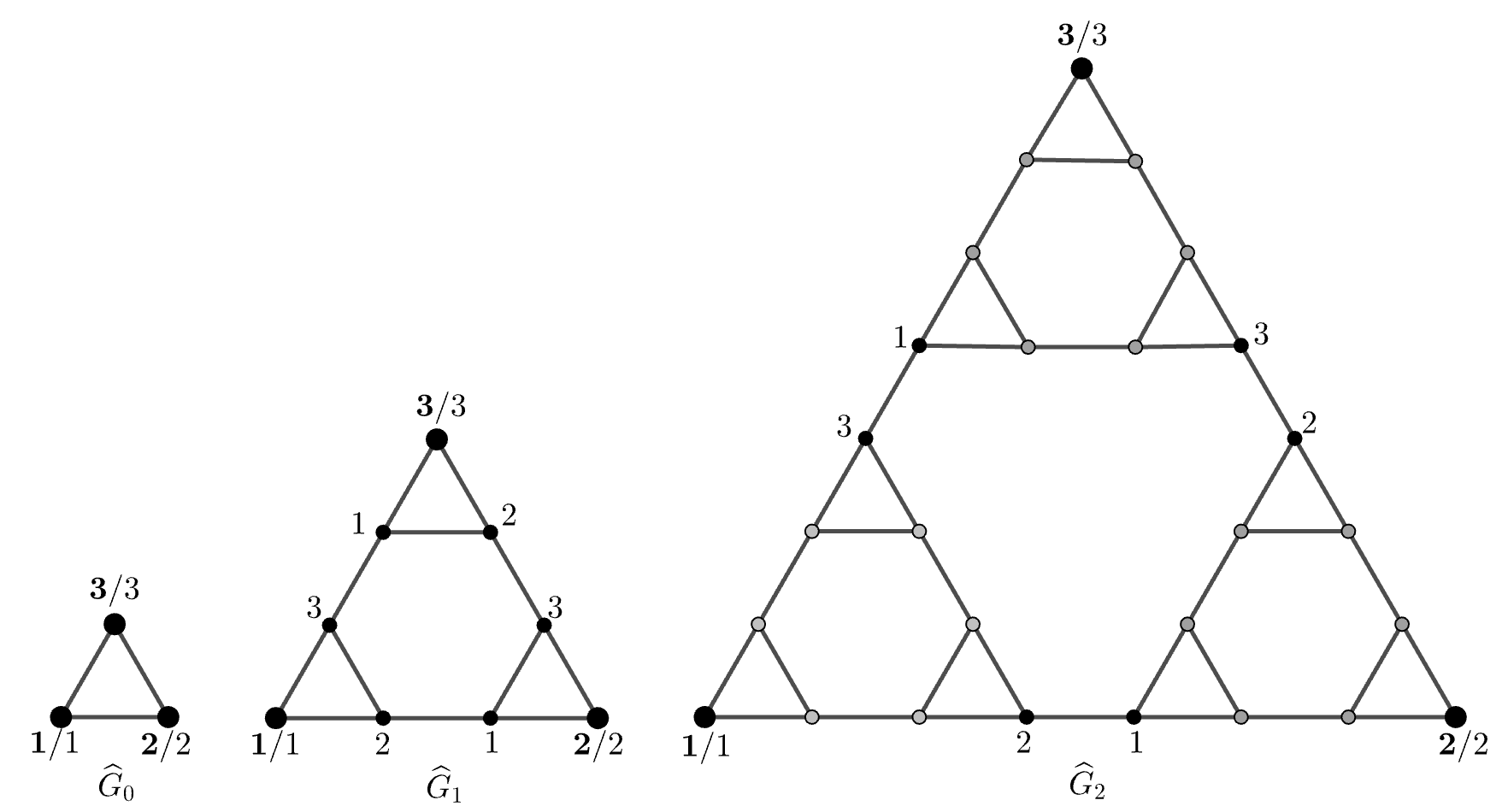}
\end{overpic}
\begin{overpic}[width=\textwidth]{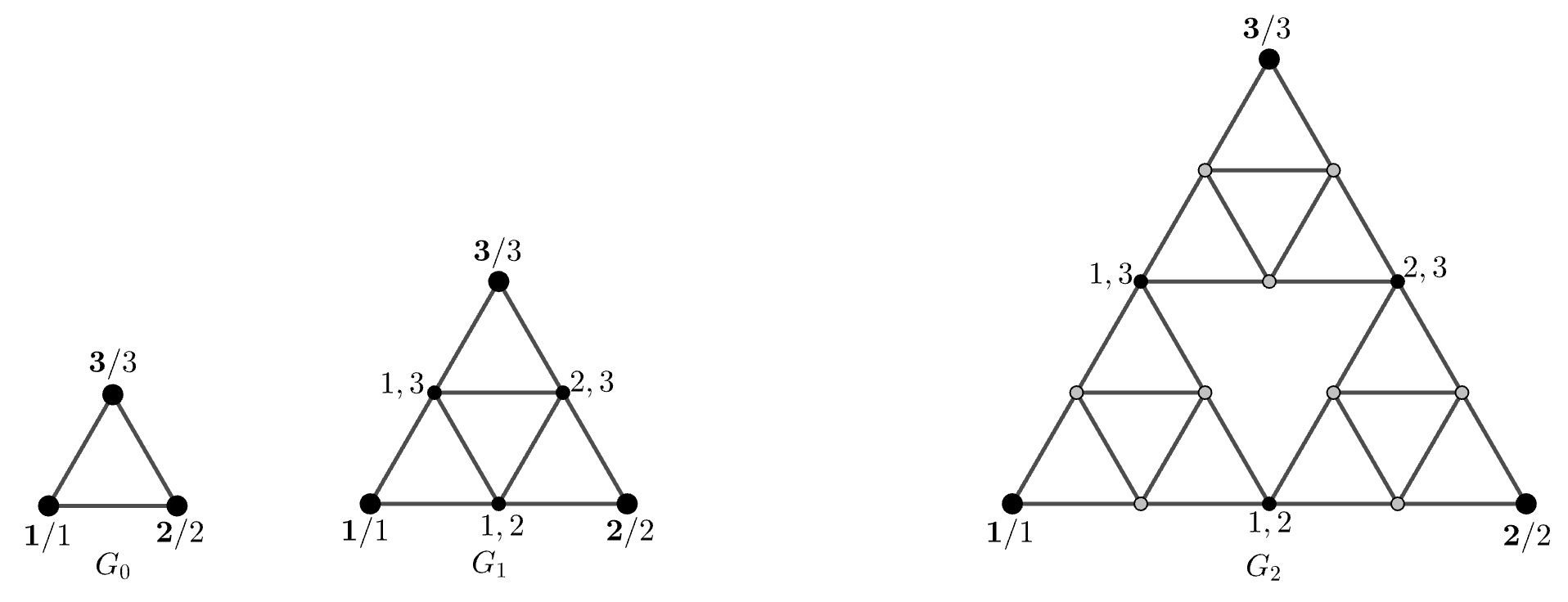}
\end{overpic}
    \caption{Illustration of the recursion for the Hanoi graphs $(\widehat{G}_{n})_{n=0,1,2}$ on top, and the Sierpi\'{n}ski gasket graphs $(G_{n})_{n=0,1,2}$ at the bottom. The thicker black vertices with their bold labels correspond to the marked vertices in each graph, labeled $1,2,3$. The black vertices all together with their normal font labels correspond to the marked vertices of the copies of the graph from the previous step.}
    \label{fig: sierpinsky-hanoi}
\end{figure}

\begin{remark}
    There are several natural canonical choices for the connecting graphs $\Sigma_e$, each depending only on the cardinality $s = \#e$.  The first option are the \emph{$s$-pods} $K_{1,s}$ -- the complete $(1,s)$-bipartite graphs. The $s$-pods are used when defining \emph{Cayley trees}. The second option are the \emph{$s$-cycles}, which are used to construct \emph{Schreier graphs}. Other natural options include the \emph{complete graphs} $K_s$ or simply the
    \emph{singletons} $K_1$. In the latter case, the corresponding marked vertices in the $s$ copies of the original graph are identified to a single vertex. This choice is, for example, used to define \emph{diamond hierarchical lattices} or other \emph{hierarchical graph systems}. Figure \ref{fig: sierpinsky-hanoi} illustrates the impact of a specific choice of the graphs $\Sigma_e$. The two sequences of graphs in the figure are obtained using the same starting graphs $G_0=\widehat G_0=K_3$ and the same reduced gluing data $(H,\Phi)$, but using different canonical choices for the connecting graphs $\Sigma_e$: the complete graphs (top row) and the singletons (bottom row). In the former case, one obtains the \emph{Hanoi graphs}, and in the latter case the \emph{Sierpi\'{n}ski gasket graphs}.
\end{remark}

\begin{definition}[Dynamics on the label set]
    A reduced gluing data $(H,\Phi)$ with parameters $m, k$ specifies a dynamical system $\Lambda=\Lambda_{(H,\Phi)}: \{1,\ldots, k\}\to \{1,\ldots, k\}$ on the labels: namely, we send each label $j\in \{1,\dots, k\}$ to $\Lambda(j)=\ell(\Phi(j))$. A label $j$ is called \emph{periodic} if $\Lambda^{p}(j)=j$ for some $p\geq 1$. Clearly, every label is eventually mapped to a periodic one by $\Lambda$.

    The number $\#\Phi(j)$ is called the \emph{local degree} of a label $j$ and the number $\#\Phi(j)-1$ is called the \emph{multiplicity} of $j$.   We say that $j$ is \emph{critical} if $\#\Phi(j)\geq 2$. Similarly to ramification portraits for rational maps, we define the \emph{portrait} of the reduced gluing data $(H,\Phi)$: it is a directed graph with the vertex set $\{1,\dots, k\}$, where we put a directed edge from every $j\in \{1,\dots, k\}$ to  $\Lambda(j)=\ell(\Phi(j))$ and label it by ``$\#\Phi(j):1$'' if $\#\Phi(j)\geq 2$.

\end{definition}

    We can now introduce the technical assumptions on the gluing data that are necessary for our main result.

\begin{definition}[Non-degenerate, stable, and expanding gluing data] Suppose $(H, \Sigma, \Upsilon, \Phi)$ is a gluing data with parameters $m, k$. We say that it is:
\begin{itemize}
    \item \emph{non-degenerate} if the induced dynamical system $\Lambda=\Lambda_{(H,\Phi)}$ has no periodic critical labels;

    \item \emph{stable} if it is non-degenerate and the connecting graph $\Sigma_{\Phi(j)}$ for each periodic label $j\in \{1,\dots, k\}$ is a singleton;

    \item \emph{expanding} if there exists $n\geq 1$ such that the gluing data $(H_n, \Sigma_n, \Upsilon_n, \Phi_n)$ of the $n$-th iterate of the graph recursion $\RR=\RR_{(H, \Sigma, \Upsilon, \Phi)}$ satisfies the following condition: the edges $\Phi_n(j)$ and $\Phi_n(\ell)$ form disjoint subsets of $V(H_n)$ for all distinct labels $j,\ell \in \{1,\dots, k\}$.
\end{itemize}
\end{definition}

\begin{remark} Suppose $(H, \Sigma, \Upsilon, \Phi)$ is an arbitrary gluing data with parameters $m,k$, and
$(H, \Sigma', \Upsilon', \Phi)$ is the ``simplified'' gluing data where each graph $\Sigma_e$, $e\in E(H)$, is replaced by a singleton graph $\Sigma'_e=K_1$. Let $(G_n)_{n\geq 0}$ be a sequence of marked graphs defined by this ``simplified'' gluing data, starting with some connected graph $G_0$. Then the following are true:
\begin{enumerate}[label=(\roman*)]
    \item The gluing data $(H, \Sigma, \Upsilon, \Phi)$ is non-degenerate if and only if the vertex degrees of $G_n$ are uniformly bounded.
    \item The gluing data $(H, \Sigma, \Upsilon, \Phi)$ is expanding if and only if the marked vertices in $G_n$ get separated as $n\to \infty$. More formally, suppose $v_1^n,\dots, v_k^n$ denote the vertices labeled $1,\dots, k$ in $G_n$, respectively. We say that the \emph{marked vertices get separated} in $G_n$ (as $n\to \infty$) if for all distinct labels $j,\ell\in \{1,\dots, k\}$ the graph-distance between $v_{j}^n$ and $v_{\ell}^n$ in $G_n$ goes to $\infty$ as $n\to \infty$. In particular, expansion is determined just by the reduced gluing data $(H, \Phi)$.
\end{enumerate}
\end{remark}

\subsection{Maximally independent starting graphs}
We now provide the final technical notion needed for our main result.

Let $G$ be a graph with $k$ marked vertices labeled $1, \ldots, k$, and let $\xbf = (x_1, \ldots, x_k) \in \{0,1\}^k$ be a binary $k$-tuple. We say that an independent set $I\subset V(G)$ \emph{agrees with $\xbf$} if the following condition is satisfies for each $j \in \{1, \ldots, k\}$: the marked vertex labeled $j$ is contained in $I$ if and only if $x_j = 1$.

\begin{definition}[Maximally independent graph]
    A graph $G$ with $k$ marked vertices (labeled $1, \ldots, k$) is said to be \emph{maximally independent} if for each binary $k$-tuple $\xbf$ there is among all independent subsets of $V(G)$ that agree with $\xbf$ a unique independent subset $I(\xbf)$ of maximal cardinality and if, in addition, the following identity holds:
    \begin{equation}\label{eq:maximally independent}
    \# I(1, \ldots, 1) - \# I(0, \ldots, 0) = k.
    \end{equation}
\end{definition}

\begin{remark} The inequality
   $$
    \# I(1, \ldots, 1) - \# I(0, \ldots, 0) \leq k.
    $$
is automatically satisfied, simply because the $k$ marked vertices can be removed from the independent set $I(1, \ldots, 1)$. By applying the same argument twice it follows that condition \eqref{eq:maximally independent} implies
    $$
    \#I(x_1, \ldots, x_k) - \#I(0, \ldots, 0) = \sum_{j=1}^k x_j
    $$
for any binary $k$-tuple $(x_1, \ldots, x_k) \in \{0,1\}^k$.
\end{remark}

We close the introduction with another example.

\begin{example}
    Let us define the \emph{Chebychev recursion} as follows. Starting with a graph $G_0$ with two marked vertices labeled $1,2$, we define $G_{n+1}$ by first taking two copies of $G_n$, then identifying the vertices with labels $2$ in these copies, and afterward assigning the new labels $1$ and $2$ to the two vertices previously labeled $1$.

    If the starting graph $G_0$ is $K_2$, then $G_n$ is simply a path whose length grows exponentially as $2^{n}$. None of these paths are maximally independent, and in fact, the zeros of the corresponding independence polynomials, which all happen to lie on the negative real axis, are unbounded.

    On the other hand, if $G_0$ is a tripod $K_{1,3}$, with marked vertices at two of the three leaves, then $G_1$ is maximally independent, and our main result implies that the zeros of the independence polynomials $Z_{G_n}(\lambda)$ are uniformly bounded.
\end{example}

    The Chebychev recursion from above shows that boundedness of the zeros of the independence polynomials $Z_{G_n}(\lambda)$ for a recursive sequence of graphs $(G_n)_{n\geq 0}$ may depend on the starting graph $G_0$. It also shows that our main result is false when the assumption on $G_0$ being maximally independent is removed.

\subsection*{Acknowledgments} The first author was  partially supported by the Marie Sk\l{}odowska-Curie Postdoctoral Fellowship under Grant No.\ 101068362. The authors would also like to thank Daniel
Meyer, Nguyen-Bac Dang, Guus Regts, and Palina Salanevich for various valuable discussion and comments.

\section{Graph recursion leads to rational iteration}

In this section, we will translate the recursion $G_{n+1} = \mathcal{R}(G_n)$ into a rational dynamical system, prove the existence of an invariant manifold, and describe the dynamics on the invariant manifold in terms of the gluing data. 

\subsection{Dynamical representation of the independence polynomial}\label{ss:recursion-indep-poly}

From now on, we fix a gluing data $(H, \Sigma, \Upsilon, \Phi)$ with parameters $m,k$ and assume that $(G_n)_{n\geq 0}$ is a sequence of marked graphs obtained from a starting marked graph $G_0$ according to this gluing data (see Section~\ref{ss:rec-definition}). We also assume that $V(H)=\{w_1,\dots,w_m\}$, and denote by $G_n(1), \ldots, G_n(m)$ the $m$ copies of $G_n$ that are used to construct $G_{n+1}$. Our first goal in this section is to introduce a (polynomial) dynamical system induced by the gluing data. First, we set up some notation.

Suppose $G$ is a graph with $s\geq 1$ marked vertices $v_1, \ldots, v_s$. In the following, we will typically identify a vertex assignment $\tau: \{v_1, \ldots, v_s\} \rightarrow \{0,1\}$ to the marked vertices of $G$ with the ordered list $(\tau(v_1),\dots,\tau(v_s))$ of its values. We will also write
$$
Z_G(\lambda, \tau) = Z_G(\lambda, \tau(v_1), \ldots, \tau(v_s)) = \sum_{\sigma \sim \tau} \lambda^{\#\sigma^{-1}(1)},
$$
where $\sigma \sim \tau$ means that we only sum over those independent vertex assignments $\sigma$ on $G$ whose restriction to the marked vertices equals $\tau$. Note that $Z_G(\lambda, \tau) = 0$ whenever $\tau$ assigns $1$ to two marked vertices that are adjacent in $G$. We also observe that
$$
Z_G(\lambda) = \sum_{\tau} Z_G(\lambda, \tau),
$$
where the sum is taken over all $2^s$ possible assignments $\tau$ to the marked vertices.

We denote the $k$ marked vertices of each graph $G_n$ by $v_1^n, \ldots, v_k^n$, so that $v^n_j$ is labeled $j$ for each $j\in\{1,\dots,k\}$. Given any binary sequence $(\xbf)=(x_1,\dots, x_k)\in \{0,1\}^k$, we will use the notation
$$
(\xbf)_n = (x_1, \ldots, x_k)_n = Z_{G_n}(\lambda, x_1, \ldots, x_k),
$$
so that
\begin{equation}
\label{eq: sum variables}
Z_{G_n}(\lambda) = \sum_{(\xbf) \in \{0,1\}^k} (\xbf)_n.
\end{equation}



Clearly, every (independent) vertex assignment on $G_{n+1}$ induces (independent) vertex assignments on the copies $G_n(1), \ldots, G_n(m)$ of $G_n$ and also on each connecting graph $\Sigma_e, e\in E(H)$. This allows us to write the following recursive formula for $(\xbf)_{n+1}$:
\begin{equation}\label{eq: recursion}
    (\xbf)_{n+1} = \sum_{\Ybf=\left(\ybf(1),\dots,\ybf(m) \right)}  \left(\prod_{i=1}^m (\ybf(i))_n\right) \cdot \left( \prod_{e \in E(H)} \frac{Z_{\Sigma_e}(\lambda, \Ybf|_e, \xbf|_e)}{\lambda^{\|\Ybf|_e\|}}\right).
\end{equation}
In the above equation, we use the following notation:
\begin{itemize}
   \item Each $\left(\ybf(i) \right) = \left(y_1(i),\dots, y_k(i)\right)$, $i\in \{1,\dots, m\}$, denotes an arbitrary vertex assignment to the marked vertices $v_1^n(i),\dots v_k^n(i)$ of $G_n(i)$. That is, we sum up over all binary sequences  $\Ybf=\left(\ybf(1),\dots,\ybf(m)\right) \in \{0,1\}^{mk}$.
   \item For each $e=(w_{i_1},\dots, w_{i_s})\in E(H)$ with label $\ell=\ell(e)$, $\Ybf|_e=\big(y_\ell(i_1),\ldots ,y_\ell(i_s)\big)$ denotes the assignment to the vertices $v_\ell^n(i_1),\dots v_\ell^n(i_s)$ induced by $\Ybf$. We also denote by $\xbf|_e$ the assignment to the root of $\Sigma_e$ induced by $(\xbf)$. The latter assignment is non-empty only when $\Phi(j)= e$ for some $j\in \{1,\dots, k\}$.
    \item For each $e\in E(H)$, we define
$$Z_e(\lambda, \Ybf, \xbf) =Z_{\Sigma_e}(\lambda, \Ybf|_e, \xbf|_e) = \sum_{\sigma}{\lambda^{\#\sigma^{-1}(1)}},$$
where we sum over all independent vertex assignments $\sigma$ on $\Sigma_e$ that agree with the restrictions imposed by $\Ybf$ and $\xbf$. Namely, $\sigma$ should satisfy all of the following conditions:
\begin{enumerate}[label=(\roman*)]
    \item If $\ell=\ell(e)\in \{1,\dots, k\}$ is the label of $e$, then $\sigma(\Upsilon_e(w_i)) = y_{\ell}(i)$ for each $w_i\in e$.
    \item If $\Phi(j)=e$ for some $j\in \{1,\dots, k\}$, then $\sigma(v_e) = x_{j}$.
\end{enumerate}
Note that, by definition, $Z_e(\lambda, \Ybf, \xbf) = 0$ whenever these restrictions are inconsistent.
\item Finally, for each $e=(w_{i_1},\dots, w_{i_s})\in E(H)$ with label $\ell=\ell(e)$, we define
$$\|\Ybf|_e\|:=\#V(\Ybf|_e = 1),$$
where $V(\Ybf|_e = 1)$ denotes the subset of those elements of $\{v^n_\ell(i_1),\dots, v^n_\ell(i_s)\}$ to which $\Ybf$ assigns $1$. In other words, $\|\Ybf|_e\|$ is the $1$-norm of $\Ybf|_e$. 
\end{itemize}
Note that the denominator in the second product in \eqref{eq: recursion} compensates for the double counting of vertices in $V(\Ybf|_e = 1)$.

In this way, we have just obtained a dynamical system $F=F_{(H,\Sigma, \Upsilon, \Phi)}: \C^{(2^k)}\to \C^{(2^k)}$ given by
\[(\xbf)_n, {(\xbf)\in {\{0,1\}^k}} \mapsto  (\xbf)_{n+1}, {(\xbf)\in {\{0,1\}^k}} \]
where each of the $2^k$ variables $(\xbf)_{n+1}$ is given by a homogeneous polynomial of degree $m$ in terms of the $2^k$ variables $(\xbf)_n$; for each ${(\xbf)\in {\{0,1\}^k}}$, the respective polynomial is provided by \eqref{eq: recursion}.

To illustrate the derived recursive formula, we give below explicit expressions for $Z_e(\lambda, \Ybf, \xbf)$ for two canonical choices of the connecting graphs $\Sigma_e$.

\begin{example}
Let us assume that each connecting graph $\Sigma_e$, $e=(w_{i_1},\dots, w_{i_s})\in E(H)$, is a singleton, so that the vertices labeled $\ell=\ell(e)$ in the respective $\# e = s$ copies $G_n(i_1),\dots, G_n(i_s)$ of the graph $G_n$ are identified together to form a vertex $v$ of the graph $G_{n+1}$.
Then $Z_e(\lambda, \Ybf, \xbf)$ equals $\lambda$ in the following two cases (depending on whether the vertex $v$ is marked or not):
\begin{itemize}
    \item if $e\neq \Phi(j)$ for any $j\in\{1,\dotsm,k\}$ and $y_{i_1}(\ell)= \ldots = y_{i_s}(\ell) = 1$;
    \item if $e=\Phi(j)$ for some $j\in\{1,\dotsm,k\}$ and $x_j=y_{i_1}(\ell)= \ldots = y_{i_s}(\ell) = 1$.
\end{itemize}
In all other cases, $Z_e(\lambda, \Ybf, \xbf)$ equals $0$. 
\end{example}

\begin{example} We may similarly analyze the situation when each connecting graph $\Sigma_e$, $e=(w_{i_1},\dots, w_{i_s})\in E(H)$, is an $s$-pod with the root $v_e$ specified to be its ``center''. In this setting, the value $Z_e(\lambda, \Ybf, \xbf)$ is determined as follows:
\begin{itemize}
    \item if $e\neq \Phi(j)$ for any $j\in\{1,\dotsm,k\}$ and $V(\Ybf|_e=1)\neq \emptyset$, then $Z_e(\lambda, \Ybf, \xbf) = \lambda^{\|\Ybf|_e\|}$;
    \item if $e\neq \Phi(j)$ for any $j\in\{1,\dotsm,k\}$ and $V(\Ybf|_e=1)= \emptyset$, then $Z_e(\lambda, \Ybf, \xbf) = 1+\lambda$;
    \item if $e=\Phi(j)$ for some $j\in\{1,\dots,k\}$ and $x_j=0$, then $Z_e(\lambda, \Ybf, \xbf) = \lambda^{\|\Ybf|_e\|}$;
    \item if $e=\Phi(j)$ for some $j\in\{1,\dots,k\}$ and $x_j=1$, then $Z_e(\lambda, \Ybf, \xbf) = \lambda$ when $y_{i_1}(\ell)= \ldots = y_{i_s}(\ell) = 0$ and $Z_e(\lambda, \Ybf, \xbf) = 0$ in all other cases.
\end{itemize}
\end{example}

\subsection{Invariant manifold}\label{subsec: inv manifold}
In this and subsequent sections, we will use the notation $({\obf}) = (0, \ldots, 0)$ for the zero vector and $(\ebf_j) = (0, \ldots, 0, 1, 0, \ldots, 0)$ for the $j$-th unit vector in $\mathbb C^k$, respectively.

Our first goal is to prove the following statement.

\begin{prop}\label{prop: invariant}
Suppose the equalities
\begin{equation}\label{eq: invariant manifold}
\frac{(x_1, \ldots, x_{j-1}, 1, x_{j+1}, \ldots, x_k)_N}{(x_1, \ldots, x_{j-1}, 0, x_{j+1}, \ldots, x_k)_N} = \frac{(\ebf_j)_N}{(\obf)_N}, \quad j=1,\dots k,
\end{equation}
hold for a fixed $N=n$ and all tuples
$$
(x_1, \ldots, x_{j-1}, x_{j+1}, \ldots, x_k) \in \{0,1\}^{k-1}.
$$
Then the same is also true for $N=n+1$. In particular, the manifold in $\C^{(2^k)}$ determined by equations \eqref{eq: invariant manifold} is invariant under the dynamical system induced by the recursion \eqref{eq: recursion}.
\end{prop}

\begin{remark}
    For a fixed $N$, each of the  denominators in equations \eqref{eq: invariant manifold} may be equal to $0$ only for finitely many parameters $\lambda\in \C$. Hence, equation it is well-defined to consider these equations as equalities between rational functions in $\lambda$.
\end{remark}

\begin{proof}
    Given any assignment $(\xbf):=(x_1, \ldots x_k)$ and label $j\in \{1,\dots, k\}$, we denote by $(\xbf_{1\to j})$ and $(\xbf_{0\to j})$ the assignments obtained from $(\xbf)$ by replacing the $j$-th entry with $0$ and $1$, respectively. We will also use analogous notations for the assignments $\ybf(1), \dots, \ybf(m)$ to the marked vertices of the copies of $G_n$.

    Suppose the equations
    $$\frac{(\xbf_{1\to j})_N}{(\xbf_{0\to j})_N}=\frac{(\ebf_j)_N}{(\obf)_N}$$ are true for some $N=n$ and all $(\xbf)\in \{0,1\}^k$.

  Fix an arbitrary assignment $(\xbf) \in \{0,1\}^k$ and $j\in \{1,\dots, k\}$, and suppose that $e=\Phi(j) \in E(H)$ and $\ell=\ell(e)\in\{1,\dots,k\}$. Let us also set $I(e)=\{i:\, w_i \in e\}$. Given any $\Ybf=\left(\ybf(1),\dots,\ybf(m) \right)\in \{0,1\}^{mk}$, we can decompose each summand in \eqref{eq: recursion} as follows:
    \begin{align*}
    &\left(\prod_{i=1}^m (\ybf(i))_n\right) \cdot \left( \prod_{e' \in E(H)} \frac{Z_{\Sigma_{e'}}(\lambda, \Ybf|_{e'}, \xbf|_{e'})}{\lambda^{\|\Ybf|_{e'}\|}}\right)\\
    =&\left(\prod_{i \in I(e)} (\ybf(i))_n\right) \cdot \frac{Z_{\Sigma_{e}}(\lambda, \Ybf|_{e}, \xbf|_{e})}{\lambda^{\|\Ybf|_{e}\|}} \cdot \left(\prod_{i \notin I(e)} (\ybf(i))_n\right) \cdot \left( \prod_{e' \in E(H)\setminus\{e\}} \frac{Z_{\Sigma_{e'}}(\lambda, \Ybf|_{e'}, \xbf|_{e'})}{\lambda^{\|\Ybf|_{e'}\|}}\right).
    \end{align*}

    By the assumption,
        \[(\ybf_{1\to \ell}(i))_n=(\ybf_{0\to \ell}(i))_n \frac{(\ebf_\ell)_n}{(\obf)_n}\]
    for all $i \in I(e)$. Hence
    $$\prod_{i \in I(e)} (\ybf(i))_n = \left(\frac{(\ebf_\ell)_n}{(\obf)_n}\right)^{\|\Ybf|_e\|} \cdot  \prod_{i \in I(e)} (\ybf_{0\to \ell}(i))_n,$$
    and we may rewrite every summand in \eqref{eq: recursion} in the following way:

    \begin{align*}
    &\left(\frac{(\ebf_\ell)_n}{(\obf)_n}\right)^{\|\Ybf|_e\|} \cdot
    \frac{Z_{\Sigma_{e}}(\lambda, \Ybf|_{e}, \xbf|_{e})}{\lambda^{\|\Ybf|_{e}\|}} \cdot \left(\prod_{i \in I(e)} (\ybf_{0\to \ell}(i))_n \right) \cdot \left(\prod_{i \notin I(e)} (\ybf(i))_n\right) \cdot \left( \prod_{e' \in E(H)\setminus\{e\}} \frac{Z_{\Sigma_{e'}}(\lambda, \Ybf|_{e'}, \xbf|_{e'})}{\lambda^{\|\Ybf|_{e'}\|}}\right).
    \end{align*}

    \smallskip

    Now we subdivide the sum in \eqref{eq: recursion} over all assignments $\Ybf$ by first summing over all possible assignments $\Ybf'$ on $\bigl\{v_\ell^n(i):\, i \notin I(e)\bigl\} \cup \bigl\{v_j^n(i):\, i\in \{1,\dots,m\}, j\neq \ell\bigl\}$, and afterward summing over all possible assignments $\Ybf_e$ on $\bigl\{v_\ell^n(i):\,i \in I(e)\bigl\}$. Using the notation $\Ybf=\left(\ybf(1),\dots,\ybf(m) \right)$ for the assignment on all the marked vertices of $G_n(1),\ldots, G_n(m)$ induced by $\Ybf'$ and $\Ybf_e$, we then have:

\begin{align*}
    (\xbf)_{n+1} &= \sum_{\Ybf_e}\sum_{\Ybf'}  \left(\prod_{i=1}^m (\ybf(i))_n\right) \cdot \left( \prod_{e' \in E(H)} \frac{Z_{\Sigma_{e'}}(\lambda, \Ybf|_{e'}, \xbf|_{e'})}{\lambda^{\|\Ybf|_{e'}\|}}\right)\\
    &=\left(\sum_{\Ybf_e}\left(\frac{(\ebf_\ell)_n}{(\obf)_n}\right)^{\|\Ybf_e\|} \cdot   \frac{Z_{\Sigma_{e}}(\lambda, \Ybf_{e}, \xbf|_{e})}{\lambda^{\|\Ybf_{e}\|}}\right)\times   \\
    &\times \left( \sum_{\Ybf'}\biggl(\prod_{i \notin I(e)} (\ybf(i))_n\biggl) \cdot \biggl( \prod_{e' \in E(H)\setminus\{e\}} \frac{Z_{\Sigma_{e'}}(\lambda, \Ybf|_{e'}, \xbf|_{e'})}{\lambda^{\|\Ybf|_{e'}\|}}\biggl) \cdot \biggl(\prod_{i \in I(e)} (\ybf_{0\to \ell}(i))_n \biggl) \right).
\end{align*}

It follows that the ratio $\displaystyle \frac{(\xbf_{1\to j})_{n+1}}{(\xbf_{0\to j})_{n+1}}$ is given by

\begin{align}\label{eq: ratio dynamics}
    \frac{(\xbf_{1\to j})_{n+1}}{(\xbf_{0\to j})_{n+1}} &=
    \left(\sum_{\Ybf_e}\left(\frac{(\ebf_\ell)_n}{(\obf)_n}\right)^{\|\Ybf_e\|} \cdot   \frac{Z_{\Sigma_{e}}(\lambda, \Ybf_{e}, \xbf_{1\to j}|_{e})}{\lambda^{\|\Ybf_{e}\|}}\right)\times   \\   \nonumber
    &\times\left(\sum_{\Ybf_e}\left(\frac{(\ebf_\ell)_n}{(\obf)_n}\right)^{\|\Ybf_e\|} \cdot   \frac{Z_{\Sigma_{e}}(\lambda, \Ybf_{e}, \xbf_{0\to j}|_{e})}{\lambda^{\|\Ybf_{e}\|}}\right)^{-1}.
\end{align}
In particular, this ratio does not depend on the values of $x_1, \ldots, x_{j-1},x_{j+1},\ldots, x_{k}$, which establishes the desired identity for $N=n+1$.
\end{proof}

\medskip

We can rewrite the collection of equations~\eqref{eq: invariant manifold} to obtain a $(k+1)$-dimensional manifold in $\C^{(2^k)}$ defined by the following equations:
\begin{equation}\label{eq: invariant manifold products}
\frac{(\xbf)_n}{ (\obf)_n} = \prod_{j:\, x_j = 1} \frac{(\ebf_j)_n}{(\obf)_n}, \quad \quad (\xbf)=(x_1,\dots,x_k)\in \{0,1\}^k.
\end{equation}
Since the recursion \eqref{eq: recursion} from $\C^{(2^k)}$ to itself is given by a homogeneous polynomial map $F=F_{(H,\Sigma, \Upsilon, \Phi)}$, we can consider it as a rational self-map
$$\widehat F= \widehat F_{(H,\Sigma, \Upsilon, \Phi)}: \P^{(2^k-1)}\to \P^{(2^k-1)}$$
of $\P^{(2^k-1)}$. We recall that this rational map will likely have some indeterminacy points -- points that are mapped to $(0, \ldots, 0) \in \C^{(2^k)}$.

Now let us consider the coordinate chart in $\P^{(2^k-1)}$ given by $(\obf) \neq 0$, a chart that is equivalent to $\mathbb C^{(2^k-1)}$. Points in this coordinate chart can be represented by the variables
$$
[\xbf]_n=[x_1, \ldots, x_k]_n := \frac{(\xbf)_n}{ (\obf)_n}=\frac{(x_1, \ldots, x_k)_n}{ (0, \ldots, 0)_n}.
$$
In these coordinates, the invariant manifold \eqref{eq: invariant manifold products} is given by the following equations:
\begin{equation}\label{eq: invariant projective manifold}
[\xbf]_n=[x_1, \ldots, x_k]_n = \prod_{j:\, x_j=1} [{\bf e_j}]_n, \quad \quad (\xbf)=(x_1,\dots,x_k)\in \{0,1\}^k \setminus (\obf).
\end{equation}
Equations \eqref{eq: invariant projective manifold} define a $k$-dimensional invariant manifold $\MM_{\ebf}$, but now in $\P^{(2^k-1)}$. We note that this manifold is a graph over the $k$-dimensional subspace spanned by the variables $\{[{\bf e_j}]_n\}_{j=1, \ldots, k}$. Moreover, it is independent of the gluing data $(H,\Sigma,\Upsilon, \Phi)$.

We can now use \eqref{eq: ratio dynamics} to describe the induced dynamics of $\widehat F$ on the invariant manifold $\MM_\ebf$ in the chosen coordinate chart in $\P^{2^k-1}$. Namely, we get the following formulas for $\{[{\bf e_j}]_{n+1}\}_{j=1, \ldots, k}$:
 \begin{equation}\label{eq: basis vectors dynamics projective}
     [\ebf_j]_{n+1} = \frac{{\displaystyle \sum_{\Ybf_e}[\ebf_\ell]_n^{\|\Ybf_e\|} \cdot   \frac{Z_{\Sigma_{e}}(\lambda, \Ybf_{e}, \ebf_j|_{e})}{\lambda^{\|\Ybf_{e}\|}}}}
     {\displaystyle {\sum_{\Ybf_e}[\ebf_\ell]_n^{\|\Ybf_e\|} \cdot   \frac{Z_{\Sigma_{e}}(\lambda, \Ybf_{e}, \obf|_{e})}{\lambda^{\|\Ybf_{e}\|}}}}.
 \end{equation}
Here, $e= \Phi(j)$, $\ell=\Lambda(j)=\ell(\Phi(j))$, and the sums are taken over all assignments $\Ybf_e$ to the marked vertices with label $\ell$ of those copies of $G_n$ that are joined by the graph $\Sigma_e$.

\begin{example}\label{ex: basis vectors dynamics projective singletons}
    Suppose for some $j\in \{1,\dots, k\}$ the connecting graph $\Sigma_{\Phi(j)}$ is a singleton. Then \eqref{eq: basis vectors dynamics projective} simply gives:
 \begin{equation}\label{eq: basis vectors dynamics projective singletons}
     [\ebf_j]_{n+1} = \frac{1}{\lambda^{\#\Phi(j)-1}}[\ebf_{\Lambda(j)}]_n^{\# \Phi(j)} .
 \end{equation}
 In particular, if $\#\Phi(j)=1$, then $[\ebf_j]_{n+1}=[\ebf_{\Lambda(j)}]_{n}$.
\end{example}

    Equation \eqref{eq: basis vectors dynamics projective} implies that the image $\widehat F(\MM_\ebf)$ of the invariant manifold lies within a submanifold of $\MM_\ebf$ of dimension $\#\Lambda(\{1,\dots,k\})$. More precisely, this submanifold is a graph over the subspace spanned by the variables $\bigl\{[\ebf_\ell]_n: \ell\in \Lambda(\{1,\dots,k\})\bigl\}$. Hence, for all sufficiently large iterates $n$, the image $\widehat F^n(\MM_\ebf)$ lies within a $k_0$-dimensional invariant submanifold $\MM_0$ of $\MM_\ebf$, where $k_0$ is the number of periodic labels $\ell\in \{1,\dots, k\}$. Combining this with the discussion in Example~\ref{ex: basis vectors dynamics projective singletons}, we get the following corollary.

\begin{corollary}\label{cor: periodic submanfold}
    Suppose the gluing data $(H,\Sigma, \Upsilon, \Phi)$ is stable. Then, in a finite number of iterations, the dynamics of $\widehat F$ projects the invariant manifold $\MM_\ebf$, defined by \eqref{eq: invariant projective manifold}, onto a submanifold $\MM_0 \subset \MM_\ebf$, which is a graph over a subspace spanned by the subcollection of the variables $[{\ebf_\ell}]_n$ corresponding to the periodic labels $\ell\in \{1,\ldots,k\}$. The dynamics of $\widehat F$ acts on $\MM_0$ by permuting these variables, and is therefore periodic. 
\end{corollary}

We will refer to the submanifold $\MM_0$ as the \emph{periodic submanifold} of $\widehat F$. When the period is $1$ we will call $\MM_0$ the \emph{fixed submanifold}.



\section{The fixed submanifold}

In this section we will prove that the periodic submanifold is superattracting, under the assumption that the gluing data $(H,\Sigma,\Upsilon, \Phi)$, with parameters $m,k$, is stable and expanding. Then by replacing the graph recursion $\RR_{(H,\Sigma,\Upsilon, \Phi)}$ with its sufficiently high iterate and redefining accordingly the initial gluing data $(H,\Sigma,\Upsilon, \Phi)$, we can make the following assumptions (up to renaming the labels):
\begin{enumerate}[label=(FM\arabic*)]
    \item\label{assump_1} There exists a (unique) $k_0\in\{1,\dots, k\}$ such that $\Lambda(j) = j$ for $j \le k_0$, and such that $\Lambda(j) \in \{1, \ldots, k_0\}$ for all $j > k_0$.
    \item\label{assump_2} All edges $\Phi(j)$ with $j\in \{1,\dots, k_0\}$ consist of a single vertex in $H$.
    \item\label{assump_3} The connecting graphs $\Sigma_{\Phi(j)}$ with $j\in \{1,\dots, k_0\}$ are all singletons.
    \item\label{assump_4}
    For any two distinct labels $j, \ell \in \{1, \ldots, k\}$, the corresponding edges $\Phi(j)$ and $\Phi(\ell)$ in $H$ are disjoint as subsets of $V(H)$, that is, the marked vertices joined by the graphs $\Sigma_{\Phi(j)}$ and $\Sigma_{\Phi(\ell)}$ lie in two disjoint collections of copies of the current graph (in each step of the recursive procedure).
\end{enumerate}

Let $\widehat F= \widehat F_{(H,\Sigma, \Upsilon, \Phi)}$ be the rational self-map of $\P^{(2^k-1)}$ induced by the recursion \eqref{eq: recursion}. Corollary~\ref{cor: periodic submanfold} implies that, under the assumptions above,  the invariant manifold $\MM_\ebf$ is $k$-dimensional, and it is projected (by the first iterate of $\widehat F$) onto a $k_0$-dimensional submanifold $\MM_0$, which is a graph over the variables $
\{[\ebf_j]_{n}\}_{j=1,\dots, k_0}$. Moreover, $\widehat F$ fixes $\MM_0$ pointwise, hence we refer to $\MM_0$ as the \emph{fixed submanifold} (of the gluing data $(H,\Sigma, \Upsilon, \Phi)$). Our goal in this section is to show the following result for $\MM_0$.

\begin{theorem}\label{thm: normally attracting}
    Suppose the gluing data is stable and expanding. Then the fixed submanifold $\MM_0 \subset \P^{(2^k-1)}$ is normally superattracting for some iterate $\widehat F^p$, $p\in \N$.  More precisely, the following three statements hold for the Jacobian matrix $J=\Jac_{\widehat F^p} (\xi)$ at any given point $\xi\in \MM_0$ of the fixed submanifold:
\begin{enumerate}[label=(\roman*)]
    \item\label{item: Jacobian i} the spectrum of $J$ equals $\{0,1\}$;
    \item\label{item: Jacobian ii} the (generalized) eigenspace associated to the eigenvalue $1$ is $k_0$-dimensional, and it is spanned by the tangent vectors to $\MM_0$ at $\xi$;
    \item\label{item: Jacobian iii} the dimension of the kernel of $J$ is $2^k-1-k_0$.
\end{enumerate}
\end{theorem}


Let us recall the defining equations for the invariant manifold $\MM_\ebf$:
\begin{equation}\label{eq:defining equations}
[\xbf]_n = \prod_{j:\, x_j = 1} [{\bf e_j}]_n, \quad [\xbf]=[x_1,\dots, x_k]\in \{0,1\}^k\setminus [\obf].
\end{equation}
We also note that there are exactly $2^{k} - 1 - k$ binary $k$-tuples $[\xbf]$ that are unequal to the $0$-tuple $[\obf]$ or to a unit vector tuple $[\bf{e_j}]$, $j=1,\dots,k$.

\begin{prop}\label{prop:attracting invariant manifold}
Suppose the gluing data $(H,\Sigma,\Upsilon, \Phi)$ is expanding and satisfies condition \ref{assump_4}. If at a point $\xi\in \P^{(2^k-1)}$ the defining equations~\eqref{eq:defining equations} are satisfied up to an additive error of order $O(\epsilon)$, then at $\widehat F(\xi)$ the defining equations are satisfied up to order $O(\epsilon^2)$.
\end{prop}

\begin{proof}[Proof of Theorem \ref{thm: normally attracting} given Proposition~\ref{prop:attracting invariant manifold}]
     By passing to an iterate, we may assume that the original gluing data satisfies conditions \ref{assump_1}--\ref{assump_4}. Since $\MM_0$ is fixed by the respective dynamics, the fact that the tangent vectors to $\MM_0$ at $\xi$ are unitary eigenvectors of the Jacobian matrix $J=\Jac_{\widehat F}(\xi)$ is immediate. By passing to even higher iterate, we may assume that the algebraic and geometric multiplicities of the eigenvalue $0$ coincide. So what remains is to determine the dimension of the kernel of $J$.

     Let  $\xi \in \MM_0$, and suppose for the purpose of a contradiction that there exists an eigenvector $v$ of the Jacobian matrix $J(\xi)$ that is not a tangent vector to $\MM_0$ at $\xi$, and with corresponding eigenvalue $\lambda$ unequal to $0$. Then clearly $v$ cannot be a tangent vector to $\MM_\ebf$ either, since the invariant manifold is projected onto the fixed manifold. By identifying the tangent space with $\mathbb C^{2^k-1}$ we obtain
     $$
     \widehat{F}(\xi+ \epsilon \cdot v) = \xi +  \epsilon \lambda\cdot v + O(\epsilon^2).
     $$
     Since $v$ is not a tangent vector to $\MM_\ebf$ and $\lambda \neq 0$, this contradicts Proposition~\ref{prop:attracting invariant manifold}.
\end{proof}

In the proof of Proposition~\ref{prop:attracting invariant manifold}, we will follow the notation used in Section~\ref{subsec: inv manifold}. In particular, we assume that $V(H)=\{w_1,\dots, w_m\}$ and set $I(e)=\{i:\, w_i \in e\}$ for a given $e\in E(H)$. We also set $$E:=\{\Phi(j): \, j=1,\dots, k\} \quad \text{and} \quad  I(E):= \bigcup_{e\in E} I(e).$$
    Furthermore, given a non-zero tuple  ${[\xbf]} = [x_1, \ldots, x_k] \in \{0,1\}^k$, we denote by $\chi_\xbf$ the unit vector in $\C^{(2^k-1)}$ whose $[\xbf]_n$-coordinate equals $1$.

\begin{proof}[Proof of Proposition~\ref{prop:attracting invariant manifold}]
    We consider one of the defining equations in~\eqref{eq:defining equations}, and we wish to prove that its partial derivatives with respect to any of the variables $[\ubf]_n$ vanish. We first claim that this is equivalent to the statement that the vector
    $$
    v_{\bf x} = \sum_{j:\, x_j = 1} \frac{1}{[\ebf_{j}]_{n+1}} \cdot \chi_{\ebf_{j}} - \frac{1}{\prod_{j:\, x_j = 1} [ \ebf_{j}]_{n+1}} \cdot \chi_{\bf x}.
    $$
    is a left-eigenvector of the Jacobian matrix $J=J(\xi)$, which is what we will prove.

     To see the equivalence, note that the vector $v_\xbf$ is a left-eigenvector of $J$ if and only if the product of $v_\xbf$ with any column vector of the Jacobian matrix $J$ is zero. A column of the Jacobian matrix $J$ is obtained by taking the derivative of each coordinate function $[\ybf]_{n+1}$ (with $[\ybf] \neq [\obf]$) with respect to a fixed coordinate, which we denote by $[\ubf]_n = [u_1, \ldots, u_k]_n$. We therefore need to show that
    \begin{equation}\label{eq: co-eigenvector idenity}
    \sum_{j:\, x_j = 1} \frac{1}{[\ebf_{j}]_{n+1}}\cdot \frac{\partial}{\partial [\ubf]_n} [\ebf_{j}]_{n+1} - \frac{1}{\prod_{j:\, x_j = 1} [ \ebf_{j}]_{n+1}} \cdot \frac{\partial}{\partial [\ubf]_n} [\xbf]_{n+1} = 0 \quad \quad     \text{for all $[\ubf]\neq [\obf]$},
    \end{equation}
    when evaluated at the point $\xi \in \MM_0$. When we multiply with $\prod_{j:\, x_j = 1} [ \ebf_{j}]_{n+1}$ we obtain that the partial derivative with respect to $[\ubf]_n$ of the defining equation for the variable $[\xbf]_{n+1}$ vanishes, which is indeed what we needed to prove.

    Let us recall that
    \begin{align*}
        [\xbf]_{n+1} = \frac{(\xbf)_{n+1}}{(\obf)_{n+1}}
        &=\frac{\displaystyle \sum_{\Ybf}  \left(\prod_{i=1}^m (\ybf(i))_n\right) \cdot \left( \prod_{e \in E(H)} \frac{Z_e(\lambda, \Ybf, \xbf)}{\lambda^{\|\Ybf|_e\|}}\right)}{\displaystyle \sum_{\Ybf}  \left(\prod_{i=1}^m (\ybf(i))_n\right) \cdot \left( \prod_{e \in E(H)} \frac{Z_e(\lambda, \Ybf, \obf)}{\lambda^{\|\Ybf|_e\|}}\right)} \\
        &=\frac{\displaystyle \sum_{\Ybf}  \left(\prod_{i=1}^m [\ybf(i)]_n\right) \cdot \left( \prod_{e \in E(H)} \frac{Z_e(\lambda, \Ybf, \xbf)}{\lambda^{\|\Ybf|_e\|}}\right)}{\displaystyle \sum_{\Ybf}  \left(\prod_{i=1}^m [\ybf(i)]_n\right) \cdot \left( \prod_{e \in E(H)} \frac{Z_e(\lambda, \Ybf, \obf)}{\lambda^{\|\Ybf|_e\|}}\right)},
   \end{align*}
   where $[\obf]_n$ is set to be $1$. For convenience, we denote the expressions in the enumerator and denominator of the last fraction by $\langle \xbf\rangle_{n+1}$ and $\langle \obf\rangle_{n+1}$, respectively; that is, we have $[\xbf]_{n+1}= \frac{\langle\xbf\rangle_{n+1}}{\langle\obf\rangle_{n+1}}$. Note that then
    $$
    \langle \xbf \rangle_{n+1} = \langle \obf\rangle_{n+1} \cdot [\xbf]_{n+1} = \langle \obf\rangle_{n+1} \cdot \prod_{j:\, x_j = 1} [\ebf_j]_{n+1}
    $$
    since $\xi$ is on the fixed submanifold $\MM_0$. We also have
    $$
    \langle \ebf_{j}\rangle_{n+1} = \langle\obf\rangle_{n+1} \cdot [\ebf_{j}]_{n+1} \quad\quad \text{for each $j=1,\dots, k$}.
    $$

    Now we simply compute the partial derivatives with respect to $[\ubf]_n$ in \eqref{eq: co-eigenvector idenity} using the quotient rule,
    and collect all the items under the common denominator $\langle\obf\rangle_{n+1}^2$. After setting the respective numerator to be equal to zero (and using the two equations  above for $\langle \xbf\rangle_{n+1}$ and $\langle \ebf_{j}\rangle_{n+1}$), we get the following equivalent identity (after dividing by the common factor $\langle \obf\rangle_{n+1}$):
    \begin{equation}\label{eq: desired identity}
    \begin{aligned}
    &
    \sum_{j:\, x_j = 1} \frac{1}{[\ebf_{j}]_{n+1}} \cdot \frac{\partial}{\partial [\ubf]_n} \langle \ebf_{j}\rangle_{n+1} - \|\xbf\|\cdot  \frac{\partial}{\partial[\ubf]_n} \langle\obf\rangle_{n+1}=\\
    &\frac{1}{\prod_{j:\, x_j = 1} [\ebf_{j}]_{n+1}} \cdot \frac{\partial}{\partial [\ubf]_n} \langle\xbf\rangle_{n+1} -
     \frac{\partial}{\partial[\ubf]_n} \langle\obf\rangle_{n+1}.
    \end{aligned}
    \end{equation}

    Thus we need to compare the partial derivatives of $\langle \ebf_{j}\rangle_{n+1}$, $\langle \xbf\rangle_{n+1}$, and $\langle\obf\rangle_{n+1}$ with respect to $[\ubf]_{n}$ (at the point $\xi\in \MM_0$). These have the following form
    \begin{equation}\label{eq: partial deriv as sum}
     \frac{\partial}{\partial[\ubf]_n}\sum_{\Ybf}  \left(\prod_{i=1}^m [\ybf(i)]_n\right) \cdot \left( \prod_{e \in E(H)} \frac{Z_e(\lambda, \Ybf, \zbf)}{\lambda^{\|\Ybf|_e\|}}\right),
    \end{equation}
    where $\zbf$ is either $\ebf_{j}$, $\xbf$, or $\obf$, respectively. We note that the three respective sums run over all possible assignments $\Ybf=\left(\ybf(1),\dots,\ybf(m) \right)$ on the marked vertices of the copies $G_1(n),\dots, G_m(n)$. We subdivide each such sum by first summing over all assignments $\Ybf_E$ on the vertex set $$\big\{v^n_{\Lambda(\ell)}(i):\, i\in I(e), e=\Phi(\ell)\in E\big\},$$
    which correspond to the marked vertices joined by the connecting graphs $\Sigma_e$ with $e=\Phi(\ell)$ for some $\ell\in \{1,\dots, k\}$, and afterward summing over all assignments $\Ybf'$ on the remaining marked vertices in the copies. It suffices to check the desired identity for one assignment $\Ybf'$ at a time. The use of the product rule for the partial derivative $\frac{\partial}{\partial[\ubf]_n}$ of each of the respective summands results in a sum of $m$ terms
    $$ \left( \prod_{e \in E(H)} \frac{Z_e(\lambda, \Ybf, \zbf)}{\lambda^{\|\Ybf|_e\|}}\right)\cdot  \left(\prod_{s \neq i} [\ybf(s)]_n\right) \cdot \frac{\partial}{\partial[\ubf]_n} [\ybf(i)]_n,$$
    one for each $i=1,\dots, m$. Now we observe that the partial derivative $\frac{\partial}{\partial [\ubf]_n} [\ybf(i)]_n$ equals $0$ unless $[\ybf(i)]_n = [\ubf]_n$, in which case the derivative is $1$. Hence, to check \eqref{eq: desired identity} it is enough to show the following identity for each $i\in \{1,\dots,m\}$ and all assignments $\Ybf=\Ybf'\cup \Ybf_E$ that satisfy $[\ybf(i)]_n = [\ubf]_n$:
    \begin{equation*}
    \begin{aligned}
   &  \sum_{j:\, x_j = 1} \frac{1}{[\ebf_{j}]_{n+1}} \cdot \left(\sum_{\Ybf_E}\bigg(\prod_{e\in E} \frac{Z_e(\lambda, \Ybf, \ebf_{j})}{\lambda^{\|\Ybf|_e\|}}\cdot \prod_{s \neq i, s\in I(E)} [\ybf(s)]_n\bigg) \right) \\
   & - \| \xbf \| \cdot \left(\sum_{\Ybf_E}\bigg(\prod_{e\in E} \frac{Z_e(\lambda, \Ybf, \obf)}{\lambda^{\|\Ybf|_e\|}}\cdot \prod_{s \neq i, s\in I(E)} [\ybf(s)]_n\bigg) \right)  =\\
    &
    \frac{1}{\prod_{j:\, x_j = 1} [\ebf_{j}]_{n+1}} \cdot \left(\sum_{\Ybf_E}\bigg(\prod_{e\in E} \frac{Z_e(\lambda, \Ybf, \xbf)}{\lambda^{\|\Ybf|_e\|}}\cdot \prod_{s \neq i, s\in I(E)} [\ybf(s)]_n\bigg) \right) \\
    &-\left(\sum_{\Ybf_E}\bigg(\prod_{e\in E} \frac{Z_e(\lambda, \Ybf, \obf)}{\lambda^{\|\Ybf|_e\|}}\cdot \prod_{s \neq i, s\in I(E)} [\ybf(s)]_n\bigg) \right) .
    \end{aligned}
    \end{equation*}
   Let us subdivide each sum over $\Ybf_E$ above into $k$ sums, each summing over all assignments $\Ybf_e$ with $e=\Phi(\ell)$, $\ell=1,\dots, k$. By assumption~\ref{assump_4}, we can rewrite the above identity as follows:
   \begin{equation}\label{eq: desired identity copy wise}
       \begin{aligned}
    &  \sum_{j:\, x_j = 1} \frac{1}{[\ebf_{j}]_{n+1}} \cdot \left(\prod_{e\in E}\sum_{\Ybf_e} \bigg(\frac{Z_e(\lambda, \Ybf_e, \ebf_{j})}{\lambda^{\|\Ybf_e\|}}\cdot \prod_{s \neq i, s\in I(e)} [\ybf(s)]_n\bigg) \right) \\
    & - \| \xbf \| \cdot \left(\prod_{e\in E}\sum_{\Ybf_e} \bigg(\frac{Z_e(\lambda, \Ybf_e, \obf)}{\lambda^{\|\Ybf_e\|}}\cdot \prod_{s \neq i, s\in I(e)} [\ybf(s)]_n\bigg) \right)  =\\
     &
     \frac{1}{\prod_{j:\, x_j = 1} [\ebf_{j}]_{n+1}} \cdot \left(\prod_{e\in E}\sum_{\Ybf_e} \bigg(\frac{Z_e(\lambda, \Ybf_e, \xbf)}{\lambda^{\|\Ybf_e\|}}\cdot \prod_{s \neq i, s\in I(e)} [\ybf(s)]_n\bigg) \right) \\
     &-\left(\prod_{e\in E}\sum_{\Ybf_e} \bigg(\frac{Z_e(\lambda, \Ybf_e, \obf)}{\lambda^{\|\Ybf_e\|}}\cdot \prod_{s \neq i, s\in I(e)} [\ybf(s)]_n\bigg) \right).
     \end{aligned}
     \end{equation}

    We now decompose $E$ as the disjoint union $E_0\sqcup E_1$, where
    $$E_0:=\{\Phi(\ell): x_\ell=0\} \quad \text{and} \quad E_1:=\{\Phi(\ell): x_\ell=1\}.$$
    Then each of the products in \eqref{eq: desired identity copy wise} has the following common factor:
    \begin{equation*}
    \prod_{e\in E_0}\sum_{\Ybf_e} \bigg(\frac{Z_e(\lambda, \Ybf_e, \obf)}{\lambda^{\|\Ybf_e\|}}\cdot \prod_{s \neq i, s\in I(e)} [\ybf(s)]_n\bigg).
    \end{equation*}
    After canceling it out, we arrive to the following identity:
    \begin{equation}\label{eq: desired identity copy wise on ones}
       \begin{aligned}
    &  \sum_{j:\, x_j = 1} \frac{1}{[\ebf_{j}]_{n+1}} \cdot \left(\prod_{e\in E_1}\sum_{\Ybf_e} \bigg(\frac{Z_e(\lambda, \Ybf_e, \ebf_{j})}{\lambda^{\|\Ybf_e\|}}\cdot \prod_{s \neq i, s\in I(e)} [\ybf(s)]_n\bigg) \right) \\
    & - \| \xbf \| \cdot \left(\prod_{e\in E_1}\sum_{\Ybf_e} \bigg(\frac{Z_e(\lambda, \Ybf_e, \obf)}{\lambda^{\|\Ybf_e\|}}\cdot \prod_{s \neq i, s\in I(e)} [\ybf(s)]_n\bigg) \right)  =\\
     &
     \frac{1}{\prod_{j:\, x_j = 1} [\ebf_{j}]_{n+1}} \cdot \left(\prod_{e\in E_1}\sum_{\Ybf_e} \bigg(\frac{Z_e(\lambda, \Ybf_e, \xbf)}{\lambda^{\|\Ybf_e\|}}\cdot \prod_{s \neq i, s\in I(e)} [\ybf(s)]_n\bigg) \right) \\
     &-\left(\prod_{e\in E_1}\sum_{\Ybf_e} \bigg(\frac{Z_e(\lambda, \Ybf_e, \obf)}{\lambda^{\|\Ybf_e\|}}\cdot \prod_{s \neq i, s\in I(e)} [\ybf(s)]_n\bigg) \right).
     \end{aligned}
     \end{equation}

 Using \eqref{eq: basis vectors dynamics projective} and the fact that $\xi\in \MM_0$, we have that the following is true for each $j$:
 \begin{equation}\label{eq: no u in the fan}
    \begin{aligned}
        [\ebf_j]_{n+1} &=
        \frac{\displaystyle \sum_{\Ybf_{\Phi(j)}}[\ebf_{\Lambda(j)}]_n^{\|\Ybf_{\Phi(j)}\|} \cdot   \frac{Z_{\Phi(j)}(\lambda, \Ybf_{\Phi(j)}, \ebf_j)}{\lambda^{\|\Ybf_{\Phi(j)}\|}}}
     {\displaystyle {\sum_{\Ybf_{\Phi(j)}}[\ebf_{\Lambda(j)}}]_n^{\|\Ybf_{\Phi(j)}\|} \cdot   \frac{Z_{\Phi(j)}(\lambda, \Ybf_{\Phi(j)}, \obf)}{\lambda^{\|\Ybf_{\Phi(j)}\|}}}\\
        &=\frac{\displaystyle \sum_{\Ybf_{\Phi(j)}}  \bigg(\frac{Z_{\Phi(j)}(\lambda, \Ybf_{\Phi(j)}, \ebf_j)}{\lambda^{\|\Ybf_{\Phi(j)}\|}}\cdot \prod_{s\in I(\Phi(j))} [\ybf(s)]_n\bigg)}{\displaystyle \sum_{\Ybf_{\Phi(j)}}  \bigg(\frac{Z_{\Phi(j)}(\lambda, \Ybf_{\Phi(j)}, \obf)}{\lambda^{\|\Ybf_{\Phi(j)}\|}}\cdot \prod_{s\in I(\Phi(j))} [\ybf(s)]_n\bigg)}.
   \end{aligned}
   \end{equation}
  It follows that when $i\notin I(\Phi(j))$ for $j$ with $x_j=1$, then:
    \begin{equation*}
    \begin{aligned}
& \frac{1}{[\ebf_{j}]_{n+1}} \cdot \left(\prod_{e\in E_1}\sum_{\Ybf_e} \bigg(\frac{Z_e(\lambda, \Ybf_e, \ebf_{j})}{\lambda^{\|\Ybf_e\|}}\cdot \prod_{s \neq i, s\in I(e)} [\ybf(s)]_n\bigg) \right) \\&- \prod_{e\in E_1}\sum_{\Ybf_e} \bigg(\frac{Z_e(\lambda, \Ybf_e, \obf)}{\lambda^{\|\Ybf_e\|}}\cdot \prod_{s \neq i, s\in I(e)} [\ybf(s)]_n\bigg) =0
    \end{aligned}
    \end{equation*}
  and
     \begin{equation*}
    \begin{aligned}
& \frac{1}{[\ebf_{j}]_{n+1}} \cdot \sum_{\Ybf_{\Phi(j)}} \bigg(\frac{Z_{\Phi(j)}(\lambda, \Ybf_{\Phi(j)}, \xbf)}{\lambda^{\|\Ybf_{\Phi(j)}\|}}\cdot \prod_{s \neq i, s\in I(\Phi(j))} [\ybf(s)]_n\bigg)  \\&- \sum_{\Ybf_{\Phi(j)}} \bigg(\frac{Z_{\Phi(j)}(\lambda, \Ybf_{\Phi(j)}, \obf)}{\lambda^{\|\Ybf_{\Phi(j)}\|}}\cdot \prod_{s \neq i, s\in I(\Phi(j))} [\ybf(s)]_n\bigg) =0.
    \end{aligned}
    \end{equation*}
 The desired equality \eqref{eq: desired identity copy wise on ones} now immediately follows since $i\in I(\Phi(j_0))$ for at most one $j_0\in \{1,\dots, k\}$ by assumption~\ref{assump_4}.
\end{proof}

\section{Boundedness of zeros}

We now finish the proof our main result.

\begin{prop}\label{prop: near infinity}
    Let $(H,\Sigma,\Upsilon, \Phi)$ be a stable and expanding gluing data with parameters $m,k$, let $\RR=\RR_{(H,\Sigma,\Upsilon, \Phi)}$ be the respective graph recursion, and let $\widehat F = \widehat F_{(H,\Sigma,\Upsilon, \Phi)}: \P^{(2^k-1)}\to \P^{(2^k-1)}$ be the dynamical system on projective space induced by the recursion \eqref{eq: recursion}. Suppose also that $G_0$ is a maximally independent starting graph (with $k$ marked vertices labeled $1,\dots, k$), and let $z_0 = z_0(\lambda) \in \P^{(2^k-1)}$ be the induced starting value.

    Given any $\epsilon >0$, there is $C(\epsilon)>0$ such that the following hols:  If $|\lambda|> C(\epsilon)$, then the $\widehat F$-orbit $(z_n)_{n\ge0}$ converges to a periodic orbit on the periodic submanifold $\MM_0$ that is $\epsilon$-close to the point $[0:\ldots : 0: 1] \in \mathbb P^{(2^k-1)}$.
\end{prop}
\begin{proof}
    Without loss of generality, we may replace $\widehat{F}$ with a sufficiently high iterate so that all periodic labels become fixed under the induced dynamics. By Theorem~\ref{thm: normally attracting}, we may further assume that the fixed submanifold $\MM_0$ is normally superattracting.

    Next, we introduce new coordinates:
    $$
    \llparenthesis x_1, \ldots, x_k \rrparenthesis_n := \lambda^{-(x_1+ \cdots + x_k)} \cdot (x_1, \ldots, x_k)_n.
    $$
    It is easy to see that the recursion \eqref{eq: recursion} may then be rewritten in the following way:
    \begin{equation}\label{eq: recursion rescaled}
    \llparenthesis\xbf\rrparenthesis_{n+1} = \lambda^{-(x_1+ \cdots + x_k)} \cdot\sum_{\Ybf=\left(\ybf(1),\dots,\ybf(m) \right)}  \left(\prod_{i=1}^m \llparenthesis\ybf(i)\rrparenthesis_n\right) \cdot \left( \prod_{e \in E(H)} Z_{\Sigma_e}(\lambda, \Ybf|_e, \xbf|_e)\right).
\end{equation}
    Indeed, in the new coordinates, the marked vertices in the copies joined by each connecting graph are no longer double counted, since they are ignored in the variables $\llparenthesis \ybf(i)\rrparenthesis_n$.

    Let $\widehat\F:\P^{(2^k-1)}\to\P^{(2^k-1)}$ be the map $\widehat F$ in these new coordinates. Note that the defining equations for the induced fixed manifold of $\widehat\F$ remain the same. It follows that $\MM_0$ remains normally superattracting with respect to $\widehat\F$.

    Since $G_0$ is maximally independent, the starting value $z_0(\lambda)$ converges in the new coordinates to the point $[1: \ldots: 1]\in \P^{(2^k-1)}$ as $\lambda$ converges to infinity. Note that the point $[1: \ldots: 1]$ lies on the fixed manifold $\MM_0$. It follows that, for any $\epsilon'>0$,   there exists a $\delta(\epsilon')$-neighborhood of $[1: \ldots: 1]$ in $\P^{(2^k-1)}$ that is contracted under $\widehat\F$ to the $\epsilon'$-neighborhood of $[1: \ldots: 1]$ in the fixed manifold $\MM_0$. Hence, for $|\lambda|>C(\epsilon')$, the starting value $z_0(\lambda)$ is in the $\delta(\epsilon')$-neighborhood of $[1: \ldots: 1]$, and thus the $\widehat \F$-orbit of $z_0(\lambda)$ converges to a fixed point within the $\epsilon'$- neighborhood of $[1: \ldots: 1]$ in the fixed manifold $\MM_0$. After changing back to the original coordinates we obtain the desired statement.
\end{proof}

Our final claim in the main result is now an immediate consequence.

\begin{corollary}
    Let $(H,\Sigma,\Upsilon, \Phi)$ be a stable and expanding gluing data with parameters $m,k$, $\RR=\RR_{(H,\Sigma,\Upsilon, \Phi)}$ be the respective graph recursion, and $(G_n)_{n\geq 0}$ be the sequence of marked graphs defined by $\RR$ starting with a maximally independent  graph $G_0$. Then the zeros of the independence polynomials $Z_{G_n}(\lambda)$ are uniformly bounded.
\end{corollary}
\begin{proof}
    Let $z_0=z_0(\lambda)$ be the initial value in $\P^{(2^k-1)}$, and $(z_n)_{n\geq 0}$ be the respective $\widehat F$-orbit.
    Recall that
    $$
    Z_{G_n}(\lambda) = \sum_{(\xbf)\in \{0,1\}^k} (\xbf)_n = (\ibf)_n\cdot\left(1+ \sum_{(\xbf)\neq (\ibf)} \frac{(\xbf)_n}{(\ibf)_n}\right),
    $$
    where $(\ibf)=(1,\dots, 1)$.

    By Proposition~\ref{prop: near infinity}, the points $z_n$ converge to a periodic orbit that lies arbitrarily close to $[0: \ldots: 0: 1]\in \P^{(2^k-1)}$ for $\lambda$ near infinity. In fact, it follows from the proof of Proposition~\ref{prop: near infinity} that the entire orbit lies arbitrarily close to this periodic orbit for sufficiently large $\lambda$.
    It follows that $Z_{G_n}(\lambda)$ cannot be zero for $|\lambda|$ sufficiently large, which completes the proof.
\end{proof}

\bibliographystyle{alpha}
\bibliography{article}

\newcommand{\etalchar}[1]{$^{#1}$}
\begin{thebibliography}{dBBG{\etalchar{+}}24}

\bibitem[BGG{\v{S}}20]{BezakovaEtAl}
I.~Bez\'akov\'a, A.~Galanis, L.A. Goldberg, and D.~{\v{S}}tefankovi{\v{c}}.
\newblock Inapproximability of the independent set polynomial in the complex
  plane.
\newblock {\em SIAM J. Comput.}, 49(5):STOC18--395--STOC18--448, 2020.

\bibitem[BLR17]{BLRI}
P.~Bleher, M.~Lyubich, and R.~Roeder.
\newblock Lee-{Y}ang zeros for the {DHL} and 2{D} rational dynamics, {I}.
  {F}oliation of the physical cylinder.
\newblock {\em J. Math. Pures Appl. (9)}, 107(5):491--590, 2017.

\bibitem[BLR20]{BLRII}
P.~Bleher, M.~Lyubich, and R.~Roeder.
\newblock Lee-{Y}ang-{F}isher zeros for the {DHL} and 2{D} rational dynamics,
  {II}. {G}lobal pluripotential interpretation.
\newblock {\em J. Geom. Anal.}, 30(1):777--833, 2020.

\bibitem[CR21]{chio2021chromatic}
I.~Chio and R.~Roeder.
\newblock Chromatic zeros on hierarchical lattices and equidistribution on
  parameter space.
\newblock {\em Ann. Inst. Henri Poincar\'e{} D}, 8(4):491--536, 2021.

\bibitem[dBBG{\etalchar{+}}24]{BoerEtAlApprox}
D.~de~Boer, P.~Buys, L.~Guerini, H.~Peters, and G.~Regts.
\newblock Zeros, chaotic ratios and the computational complexity of
  approximating the independence polynomial.
\newblock {\em Math. Proc. Cambridge Philos. Soc.}, 176(2):459--494, 2024.

\bibitem[dBBPR23]{BoerEtAlTorus}
D.~de~Boer, P.~Buys, H.~Peters, and G.~Regts.
\newblock On boundedness of zeros of the independence polynomial of tor.
\newblock {\em Preprint arXiv:2306.12934}, 2023.

\bibitem[DDN11]{DDN2011}
D.~D'Angeli, A.~Donno, and T.~Nagnibeda.
\newblock Partition functions of the {I}sing model on some self-similar
  {S}chreier graphs.
\newblock In {\em Random walks, boundaries and spectra}, volume~64 of {\em
  Progr. Probab.}, pages 277--304. Birkh\"auser/Springer Basel AG, Basel, 2011.

\bibitem[DGL23]{Bac2023}
N.-B. Dang, R.~Grigorchuk, and M.~Lyubich.
\newblock Self-similar groups and holomorphic dynamics: renormalization,
  integrability, and spectrum.
\newblock {\em Arnold Math. J.}, 9(4):505--597, 2023.

\bibitem[GG{\v{S}}17]{GalanisInapproximability}
A.~Galanis, L.A. Goldberg, and D.~{\v{S}}tefankovi{\v{c}}.
\newblock Inapproximability of the independent set polynomial below the
  {S}hearer threshold.
\newblock In {\em 44th {I}nternational {C}olloquium on {A}utomata, {L}anguages,
  and {P}rogramming}, volume~80 of {\em LIPIcs. Leibniz Int. Proc. Inform.},
  pages Art. No. 28, 13. Schloss Dagstuhl. Leibniz-Zent. Inform., Wadern, 2017.

\bibitem[HPR20]{HelmuthPerkinsRegts}
T.~Helmuth, W.~Perkins, and G.~Regts.
\newblock Algorithmic {P}irogov-{S}inai theory.
\newblock {\em Probab. Theory Related Fields}, 176(3-4):851--895, 2020.

\bibitem[LY52]{LeeYang}
T.D. Lee and C.N. Yang.
\newblock Statistical theory of equations of state and phase transitions. {II}.
  {L}attice gas and {I}sing model.
\newblock {\em Phys. Rev. (2)}, 87:410--419, 1952.

\bibitem[PR17]{PatelRegts}
V.~Patel and G.~Regts.
\newblock Deterministic polynomial-time approximation algorithms for partition
  functions and graph polynomials.
\newblock {\em SIAM J. Comput.}, 46(6):1893--1919, 2017.

\bibitem[PR19]{PetersRegts}
H.~Peters and G.~Regts.
\newblock On a conjecture of {S}okal concerning roots of the independence
  polynomial.
\newblock {\em Michigan Math. J.}, 68(1):33--55, 2019.

\bibitem[RL19]{Rivera}
J.~Rivera-Letelier.
\newblock Complex dynamics and the neighbor exclusion model on the {C}ayley
  tree.
\newblock
  \url{http://www.fields.utoronto.ca/talks/Complex-dynamics-and-neighbor-exclusion-model-Cayley-tree},
  2019.
\newblock Fields Center, Toronto.

\bibitem[vW23]{Willigen}
J.~van Willigen.
\newblock The hard-core model on infinite trees.
\newblock Master Thesis, available
  at~\url{https://scripties.uba.uva.nl/search?id=record_53670}, University of
  Amsterdam, 2023.

\bibitem[YL52]{YangLee}
C.N. Yang and T.D. Lee.
\newblock Statistical theory of equations of state and phase transitions. {I}.
  {T}heory of condensation.
\newblock {\em Phys. Rev. (2)}, 87:404--409, 1952.

\end{thebibliography}

\end{document}